\documentclass[letterpaper, 10 pt, conference]{ieeeconf}  
\usepackage[T1]{fontenc}
\usepackage[utf8]{inputenc}
\usepackage{xcolor}
\usepackage{mathtools}

\let\labelindent\relax
\usepackage{enumitem}

\newtheorem{theorem}{Theorem}

\newtheorem{proposition}[theorem]{Proposition}

\newtheorem{definition}[theorem]{Definition}

\allowdisplaybreaks

\IEEEoverridecommandlockouts                              

\usepackage{graphics} 
\usepackage{epsfig} 
\usepackage{amsmath} 
\usepackage{amssymb}  

\title{\LARGE \bf A mathematical framework for time-delay reservoir computing analysis}

\author{Anh Tuan Clabaut$^{1}$, Jean Auriol$^{1}$, Islam Boussaada$^{2, 3}$, and Guilherme Mazanti$^{2, 4}$
\thanks{$^{1}$Universit\'e Paris-Saclay, CNRS, CentraleSup\'elec, Laboratoire des signaux et syst\`emes, 91190 Gif-sur-Yvette, France. \newline {\tt\small firstname.lastname@l2s.centralesupelec.fr}.}\thanks{$^{2}$Universit\'e Paris-Saclay, CNRS, CentraleSup\'elec, Inria, Laboratoire des signaux et syst\`emes, 91190 Gif-sur-Yvette, France. \newline {\tt\small firstname.lastname@l2s.centralesupelec.fr}.}\thanks{$^{3}$IPSA, 94200 Ivry-sur-Seine, France.}\thanks{$^{4}$F\'ed\'eration de Math\'ematiques de CentraleSup\'elec, 91190 Gif-sur-Yvette, France.}}

\begin{document}

\maketitle
\thispagestyle{empty}
\pagestyle{empty}

\begin{abstract}
Reservoir computing is a well-established approach for processing data with a much lower complexity compared to traditional neural networks. Despite two decades of experimental progress, the core properties of reservoir computing (namely separation, robustness, and fading memory) still lack rigorous mathematical foundations. This paper addresses this gap by providing a control-theoretic framework for the analysis of time-delay-based reservoir computers. We introduce formal definitions of the separation property and fading memory in terms of functional norms, and establish their connection to well-known stability notions for time-delay systems as incremental input-to-state stability. For a class of linear reservoirs, we derive an explicit lower bound for the separation distance via Fourier analysis, offering a computable criterion for reservoir design. Numerical results on the NARMA10 benchmark and continuous-time system prediction validate the approach with a minimal digital implementation.
\end{abstract}

\section{Introduction}

Reservoir Computing (RC) is a machine learning framework inspired by Recurrent Neural Networks (RNN) that aims to reduce training complexity by fixing the hidden layer, which is then called \emph{reservoir}. RC appeared in the early 2000's with Echo State Networks \cite{jaeger2001echo} and Liquid State Machines \cite{maass2002real} and was shown to provide good performance for tasks with history-dependent inputs. Instead of training all the weights via backpropagation algorithms, the idea of RC is to reduce the hidden layer to a dynamical system with fixed weights. As in traditional neural networks, for given input signals, the outputs are 
obtained through a linear combination of the reservoir states, so that only the 
weights between the reservoir and the output layer are trained. This implies that 
a linear regression suffices to train the model, dramatically reducing 
the computational cost compared to standard RNNs~\cite{Lukosevicius2009Reservoir}.  For this system to achieve good performance despite the lack of training, the computational capacity of the dynamical system needs to be very high.

In our approach, the reservoir is a Time-Delay System (TDS). This type of reservoir was introduced in~\cite{appeltant2012reservoir} and  was shown to achieve performances comparable to state-of-the-art reservoir computing with both electronic and opto-electronic implementations. Subsequently, implementations based on photonic architectures have been performed, see for example~\cite{duport2016fully, larger2017high}. The key advantage of this model lies 
in its physical implementation: while traditional reservoirs require many neurons to 
be implemented physically, the time-delay model only requires a single physical node, 
making it particularly attractive for hardware realization.

The performance of RC systems is commonly characterized through three fundamental 
properties. The \emph{separation property}  refers to the ability of the 
reservoir to map distinct inputs onto distinguishable states. The \emph{fading memory 
property}  ensures that the influence of past inputs decays over time, so that 
the reservoir is not dominated by its initial conditions. Finally, \emph{robustness} 
captures the ability to tolerate noise without overfitting. These properties have been 
central to the RC literature since its inception, 
and their interplay is known to govern the computational capacity of the 
reservoir~\cite{dambre2012information}. In particular, best performance is often 
empirically found when the system operates at the \emph{edge of chaos}, a regime 
balancing separation and stability \cite{legenstein2007edge}. In 2021, strong links 
were pointed out between the spectrum of the linearized system around a fixed point 
and the memory capacity of the reservoir~\cite{koster2021insight}. However, despite this growing body of experimental knowledge, these three properties still 
lack rigorous mathematical foundations.  A notable exception is~\cite{dambre2012information}, which 
quantifies the information processing capacity of dynamical systems in terms of functional spaces. However, this framework does not address the input-output separation, robustness to noise, or fading memory properties from a system-theoretic perspective, nor does it exploit the specific structure of time-delay systems. To the best of our knowledge, no control-theoretic formalization of these properties has been established for time-delay reservoir computers. This paper addresses this gap by connecting reservoir computing properties to well-established concepts from control theory thereby providing both a rigorous foundation and concrete design criteria for time-delay reservoirs. The main contributions are as follows:
\begin{itemize}
    \item We introduce {formal definitions} of the separation property and 
    fading memory in terms of functional norms, grounding these notions within the 
    framework of infinite-dimensional dynamical systems.
    \item We establish a {connection between fading memory and incremental 
    input-to-state stability} ($\delta$ISS).
    \item For a class of {linear reservoirs with a single delay}, we derive an 
   {explicit lower bound for the separation distance} via Fourier analysis.
\end{itemize}

\paragraph*{Notation}
Let $n, m$ be positive integers. For any $\Delta>0$, we denote $C \coloneqq C_p([-\Delta, 0), \mathbb{R}^n)$ (resp. $C_p(\mathbb{R}_+, \mathbb{R}^n)$) the set of piecewise continuous functions defined on $[-\Delta,0)$ (resp. $\mathbb{R}_+$) with values in $\mathbb{R}^n$. The set $C$ is endowed with the norm $\|\cdot\|$ such that for any $\varphi \in C$, $\|\varphi\|=\sup_{s\in [-\Delta,0)}\lvert\varphi(s)\rvert$, where $\lvert \cdot \rvert$ is the Euclidean norm. For all $u \in C_p(\mathbb{R}_+, \mathbb{R}^n)$, $
\|u \|_{[0,t]} = \sup_{s \in [0,t]} \lvert u(s) \rvert.$ In the article, we will consider reservoirs based on retarded time-delay systems of the form
\begin{equation}  \label{eq gen}
        \dot{x}(t) = f(x_t, u(t)),\quad x_0= \phi,
\end{equation}
where $x_t$ refers to the history function $x_t\colon s \longmapsto x(t+s)$ for $s \in [-\Delta, 0)$, where $\Delta > 0$ is given. The initial conditions are taken from the set $C$. The inputs are taken in $U \coloneqq L^\infty \cap C_p(\mathbb{R}_+, \mathbb{R}^m)$ and the state $x$ take values in $\mathbb{R}^n$. Conditions on $f$ guaranteeing the well-posedness of~\eqref{eq gen} and the existence of solutions can be found in~\cite{hale2013introduction}. Given an input $u$ and an initial condition $\phi \in C$, the solution of~\eqref{eq gen} at time $t$ will be denoted by $x^{u}(\phi)(t)$ (or $x(t)$ if no confusion arises).


\section{Desired properties for a reservoir} \label{desired prop}

In this section we summarize the properties that are essential for a reservoir computer to achieve good performance. 

The \emph{generalization property} is the ability of the model, after being trained on input-output pairs, to approximate the corresponding outputs of unseen inputs. This property mainly depends on the richness of the dynamics and the optimization of the weights. As our focus is on the impact of the system parametrization, training aspects will not be developed in this work (see \cite{appeltant2012reservoir, hoerl1970ridge} for more details on training).

The capability of the reservoir to map distinct inputs into distinct outputs, even for very close signals, is called the \emph{separation property} (SP).
The SP is linked to the complexity of the reservoir dynamics~\cite{dambre2012information}, which explains why nonlinear functions have always been a key issue in neural networks. In our study, the richness of the reservoir is also captured by the infinite-dimensional state space inherent to TDS. Since input signals are typically corrupted by noise, a complementary 
property is required: \emph{robustness}. A robust reservoir interprets small difference between inputs as noise, thereby 
avoiding overfitting \cite{bishop1995training}. The contradictory aspect of these two properties implies that a trade-off must be found. 

The next property, specific to RC, is the \emph{fading memory} property 
\cite{maass2002real}. It measures how much the system “forgets” the past and requires that the influence of past inputs decays over time, 
so that the current state reflects recent inputs more than distant ones. This is 
especially critical in TDS-based reservoirs, where past information is reinjected into the system, implying that past noise accumulates with current noise. Moreover, as it is technologically demanding on real setups to prevent the devices from noise when initializing the system, we want the reservoir to ignore the information that is too influenced by its initial state. 

So far, the analysis of these properties has been mostly experimental and qualitative, with comparisons for different sets of parameters (including the choice of the nonlinearity) being done after the system is trained and the estimated outputs are computed~\cite{appeltant2012reservoir}. The present work aims to provide a more rigorous mathematical foundation for this analysis.

\section{Time-delay reservoir computing}  \label{RC SAR}

In this section, we provide more details on the different components of time-delay reservoirs.

\subsection{Time-delay systems as reservoirs}
The reservoir we consider is a dynamical system of the form \eqref{eq gen}, $u$ being the signal to be processed.  TDS are particularly well-suited 
for reservoir computing for two reasons. First, the retarded feedback term produce highly complex dynamics. Indeed, re-injecting past information drastically increases the capability of the system to mitigate the inputs, which leads to good separation and generalization~\cite{appeltant2012reservoir}. Therefore, even weak nonlinearities can achieve good performance. Secondly, the retarded feedback term strengthens the short-term memory, resulting in good computational capacity \cite{dambre2012information}. Furthermore, the 
time-delay architecture requires only a single physical node, making it 
attractive for hardware implementations \cite{larger2017high}.

\subsection{Input driving}

For the system to map the input into a high-dimensional phase space, the reservoir response is sampled multiple times, creating ``temporal'' dimensions rather than increasing the number of nodes. The sample points are called \emph{virtual nodes}, as they effectively contribute to the dimensionality of the response. 
To reproduce this process with only one node, the input signal undergoes two pre-processing steps before being injected in the system. The first one, called \emph{time multiplexing}, consists in a sample-and-hold operation 
applied to the input~$u$, producing a piecewise constant function $I$ with step 
duration $T$, called the \emph{clock cycle}. The second step, called \emph{masking procedure}, consists in multiplying the input stream by a piecewise constant $T$-periodic function, called \emph{mask}, with step duration $\theta$. The mask plays the same role as the weight matrix that usually feeds the input vector to the $N = T/\theta$ \emph{virtual nodes} with different scaling factors. 

\subsection{Timescales}
The setup involves three timescales that must be carefully tuned. The first one is the intrinsic time constant $\theta_0$ of the node. It is the time that the system takes to respond to a perturbation. 
For systems whose trajectories decay at an exponential rate, this notion is easy to define: if the solutions decay like $e^{-\frac{t}{\tau}}$, the time constant is $\tau$. The virtual node spacing $\theta$ 
should be chosen so that consecutive node responses are sufficiently 
decorrelated, and the clock cycle $T = N\theta$ should match the delay $\tau$ 
to maximize memory.  Note that resonance effects have been observed when $T$ is too close to a multiple of $\tau$, resulting in memory degradation and thus in a loss of accuracy. This phenomenon is detailed in \cite{koster2021insight}. All these constraints lead to $\theta \lesssim \theta_0 \ll T \sim \tau.$
The last remaining timescale is the quantization step. If it is too small, the complexity will be high, and if it is too large, quantization noise will reduce accuracy.
If the input is in discrete time, there is no quantization step: the values are simply held during a time $T$.

\section{Framework for reservoir analysis}  \label{math}

We now formalize the properties introduced in Section~\ref{desired prop} 
as analytical properties of dynamical systems. 

\subsection{Stability}
It has already been clarified that stability is a necessary condition for any meaningful prediction. Nevertheless, if the convergence rate is too strong, 
it causes the reservoir 
to lose the input information too quickly. For these reasons, the convergence rate must be tuned. In the linear case, stability can be assessed via the spectral abscissa~\cite{amrane2018qualitative,Michiels2014Stability}, as we will discuss in Section~\ref{lin}. For nonlinear systems, the general framework to study the stability of a TDS depending on an input stream is the \emph{Input-to-State Stability} (ISS)  \cite{yeganefarinput,chaillet2023iss}. The following result, taken from \cite{chaillet2023iss}, characterizes this concept through the formalism of Lyapunov--Krasovskii functionals.
\begin{proposition}
System \eqref{eq gen} is ISS if and only if there exist a functional $V \colon C \rightarrow \mathbb{R}$ Lipschitz on bounded sets, $\alpha_1, \alpha_2, \alpha_3, \gamma \in \mathcal{K}_\infty$ such that for every $\phi \in C$ and every $u \in U$, the corresponding solution $x$ satisfies
\begin{subequations}
    \begin{align}
        & \alpha_1(|x(t)|)  \le V(x_t) \le \alpha_2(\|x_t\|), \label{LKF-candidate}\\
        & D^{+}V(x_t) \le -\alpha_3\big(V(x_t)\big) + \gamma(\|u\|_{[0,t]}). \label{dissipation}
    \end{align}
\end{subequations}
\end{proposition}
A functional Lipschitz on bounded sets that satisfies \eqref{LKF-candidate} is called a Lyapunov--Krasovskii Functional (LKF) candidate. 

\subsection{Separation property}
Inspired by the definition given in~\cite{appeltant2012reservoir}, we quantify the SP by the $L^2$ norm of the state difference induced by two 
distinct inputs $u$ and $v$ sharing the same initial condition $\phi$:
\begin{align}
    S_{t_0, t_1, \phi}(u,v) \textstyle \coloneqq \| x^{u}(\phi) - x^{v}(\phi)\|_{L^2([t_0, t_1])}.
\end{align}

To obtain a quantity independent of the specific input pair and initial condition, 
we average over a set $F \subset C$ of initial conditions and over pairs $(u,v) \in E \subset U \times U$ satisfying $\big|\|u-v\|_{L^2} - d\big| \simeq \delta$, where $d$ is a characteristic distance and $\delta$ a precision~\cite{maass2002real}:
\begin{equation}
    \textstyle S_{d, t_0, t_1} \coloneqq \frac{1}{|E||F|} \sum_{(u,v) \in E} \sum_{\phi \in F} S_{t_0, t_1, \phi}(u,v).
\end{equation}

We expect this quantity to be large enough with respect to $\|u-v\|_{L^2}$, ensuring injectivity for the input-to-state map. A lower bound for $S_{d,t_0,t_1}$ in the linear single-delay case is derived in Section~\ref{lin}. The $L^2$ norm is chosen here for the 
orthogonality of Fourier modes, which is exploited in Section~\ref{lin}. Other choices of $L^p$ norms can be relevant too. When computing the separation from a physical setup, $t_0$ is chosen large enough for the transient due to initial conditions to have decayed. 

So far, the SP has been assessed empirically by running a large number of simulations and checking the rank of the state matrix~\cite{appeltant2012reservoir}. The $L^2$-based criterion $S_{d,t_0,t_1}$ proposed here offers a more principled alternative, as it involves the full trajectory rather than a finite collection 
of samples. From a system-theoretic perspective, requiring a uniform lower bound of the form $S_{d,t_0,t_1} \ge C d$ for all input pairs amounts to asking for \emph{exact observability} of the reservoir, i.e., the existence of a constant $C > 0$ such that
\begin{equation}
    \|\mathcal{T}u - \mathcal{T}v\|_{L^2} \geq C \|u - v\|_{L^2},
\end{equation}
where $\mathcal{T} \colon u \mapsto x^u$ denotes the input-to-state operator. 
For linear systems however, $\mathcal{T}$, the variation-of-constants formula~\cite{hale2013introduction} implies that no such uniform bound can hold over the full input space. Intuitively, high-frequency input components are attenuated by the reservoir dynamics, producing a \emph{filtering effect} that prevents separation of rapidly oscillating inputs, a phenomenon consistent with the robustness requirement of 
Section~\ref{desired prop}. A natural remedy is to restrict the analysis to inputs within a finite-dimensional frequency band on which a lower bound of $\mathcal{T}$ can be recovered. This motivates the frequency-limited analysis of Section~\ref{lin}.

Finally, this injectivity-related problem echoes the design of \emph{Kazantzis--Kravaris--Luenberger} 
(KKL) observers \cite{bernard2022observer}, where 
injectivity of a state-space immersion is recovered by increasing the dimension  of the target space. In reservoir computing, adding virtual nodes plays the same role. Interestingly, connections between KKL observers and machine learning have recently been drawn \cite{janny2021deep}, suggesting that a reservoir can be 
interpreted as a data-driven approximation of a KKL immersion. From this perspective, a TDS reservoir offers an \emph{intrinsically infinite-dimensional} immersion space, which should in principle provide richer separation guarantees than any finite-dimensional architecture. Whether this structure can be systematically exploited to enforce injectivity over a prescribed input class is an open question.

\subsection{Robustness to noise}
A robust reservoir should attenuate the effect of small input perturbations on 
the state. To deal with noise, we want to have some control on the state difference corresponding to two slightly different inputs. This can be formalized as 
\begin{align}
     |x^{u}(\phi)(t) - x^{v}(\phi)(t)| \leq \gamma({\|u-v\|_{[0,t]}}),\label{ineq_noise}
     \end{align}
with $\gamma$ a positive and non-decreasing function and where $v$ can be a noisy version of $u$. We expect $\gamma$ to be slowly increasing since, in that way, the state difference is relatively small compared to the input difference. Equation~\eqref{ineq_noise} is precisely the trade-off with the SP: a highly sensitive reservoir separates inputs well but  amplifies noise, while an overly robust one fails to distinguish genuine input 
differences. The only way to distinguish input separation from noise is to have at least a rough idea of the shape of the noise. In our analysis, we consider that high-frequency 
components of a signal can be attributed to noise \cite{bishop1995training}. Consequently, we require the reservoir to act as a low-pass filter, attenuating 
high-frequency inputs while preserving low-frequency content.

\subsection{Fading memory}
Beyond input noise, the initial condition of the reservoir carries no useful 
information about the current input and should be progressively forgotten. The initial state can indeed be strongly affected by previous entries (as it may be hard on physical setups to reinitialize the system between two inputs). This implies that we also need to be robust with respect to initial conditions. More precisely, for any $u \in U$, for any $\phi , \psi \in C$, we need $|x^{u}(\phi)(t) - x^{u}(\psi)(t)|$ to be continuously upper-bounded by a term depending on $\|\phi - \psi\|$, and this term shall be decreasing with respect to time. The Incremental Input-to-State Stability ($\delta$ISS)  unifies robustness 
to noise and fading memory into a single notion that can be characterized through sufficient conditions. 

\begin{definition}[Incremental Input-to-State Stability]
The system \eqref{eq gen} is said to be \emph{incrementally input-to-state stable ($\delta$ISS)} if there exist functions $\beta \in \mathcal{KL}$ and $\gamma \in \mathcal{K}_\infty$ such that, for all $\phi, \psi \in C$, {for all $u,v \in U$, and} for all $t \ge 0$,
\[
  \lvert x^{u}(\phi)(t) - x^{v}(\psi)(t)\rvert \le  \beta\left(\|\phi - \psi\|,t\right) + \gamma\!\left(\|u-v\|_{[0,t]}\right).  
\]
\end{definition}

The following theorem provides a Lyapunov sufficient condition for $\delta$ISS. The proof follows the classical methodology of ISS results \cite{chaillet2023iss} but we provide a self-contained proof adapted to this setting.

\begin{theorem}
If there exist a functional $V\colon C \times C \rightarrow [0, +\infty)$ {Lipschitz on bounded sets}, $\alpha_1, \alpha_2 \in \mathcal{K}_\infty$, $\sigma \in \mathcal{K}$ and $\kappa > 0$ such that for any inputs $u, \tilde u \in U$ and any $x_0, \tilde x_0 \in C$, the associated solutions of~\eqref{eq gen} (respectively denoted here by $x$ and $\tilde x$)  verify for all $t\geq 0$
    \begin{align}
    \alpha_1(|x(t)-\tilde{x}(t)|)  \le V(x_t,\tilde{x}_t) \le \alpha_2(\|x_t-\tilde{x}_t\|) \label{dISS1}, \\
    D^+ V(x_t,\tilde{x}_t) \le -\kappa V(x_t,\tilde{x}_t) + \sigma\big(\|u-\tilde{u}\|_{[0,t]}\big) \label{dISS2},
    \end{align}
then system \eqref{eq gen} is $\delta$ISS.
\end{theorem}
\begin{proof}
By \eqref{dISS2} and Grönwall's lemma,
\begin{align*}
 V(x_t,\tilde{x}_t)
&\textstyle\le e^{-\kappa t} V(x_0,\tilde{x}_0) + e^{-\kappa t}\int_0^t e^{\kappa s}\, \sigma\big(\|u-\tilde{u}\|_{[0,s]}\big)\,ds \\
& \le e^{-\kappa t} V(x_0,\tilde{x}_0) + \tfrac{1}{\kappa}\sigma\big(\|u-\tilde{u}\|_{[0,t]}\big).
\end{align*}

Using \eqref{dISS1} for $t=0$,
\[V(x_t,\tilde{x}_t) \le e^{-\kappa t} \alpha_2(\|x_0 - \tilde{x}_0 \|) + \tfrac{1}{\kappa}\sigma\big(\|u-\tilde{u}\|_{[0,t]}\big).\]
Then, using the 
left-hand side  of \eqref{dISS1} and the fact that the inverse function  of a $\mathcal{K}_\infty$ function is in $\mathcal{K}_\infty$ as well,
\begin{align*}
 |x(t) -\tilde{x}(t)|
& \le \beta \big( \|x_0 - \tilde{x}_0 \| , t \big) + \gamma(\|u-\tilde{u}\|_{[0,t]})
\end{align*}
for $\beta ( x, t) = \alpha_1^{-1} \big(2e^{-\kappa t} \alpha_2(x) \big)$, $\gamma(x) = \alpha_1^{-1} \big( \frac{2}{\kappa}\sigma(x) \big)$.
\end{proof}


The connection between $\delta$ISS and reservoir computing is the central contribution of this framework. The $\mathcal{KL}$ term $\beta$ captures  \emph{fading memory}: it must decay fast enough for the initialization transient to vanish before the first readout, yet slowly enough to preserve memory of recent 
inputs. The $\mathcal{K}_\infty$ term $\gamma$ captures \emph{robustness}: it  should grow slowly so that noise is attenuated without masking genuine input  differences. This makes the separation/robustness trade-off explicit at the  functional level: stronger contraction accelerates the decay of $\beta$ at the  cost of memory, while higher sensitivity improves the SP but steepens $\gamma$.  Finally, while finding a suitable Lyapunov--Krasovskii functional $V$ can be challenging for nonlinear reservoirs, standard constructive methodologies (such as Linear Matrix Inequalities (LMIs)) are well-established in the control literature \cite{fridman2014introduction,karafyllis2011stability,pepe2006lyapunov}.

\section{Case of a linear reservoir} \label{lin}

Although practical reservoirs rely on nonlinear dynamics, the linear case already provides valuable theoretical insight and is not without a computational interest. Indeed, while the state space of \eqref{retarded lin} is finite-dimensional in the classical sense, the solution operator acts on the infinite-dimensional space $C$, so that the reservoir response retains a rich structure even in the absence of nonlinearity \cite{hale2013introduction}. This stands in sharp contrast with linear reservoirs based on ODEs, whose computational capacity is fundamentally limited \cite{dambre2012information}. Furthermore, the linear case serves as a stepping stone toward the nonlinear setting: local behavior near a fixed point can be analyzed via linearization, and the spectral properties of the linearized system have been shown to strongly influence the memory capacity of the reservoir \cite{koster2021insight}. Linear reservoirs are also directly relevant for tasks with low nonlinearity requirements such as NARMA10 \cite{appeltant2012reservoir}, and can be combined with a nonlinear readout to increase expressivity \cite{Lukosevicius2009Reservoir}.

\begin{proposition}  \label{cv rate} \cite{hale2013introduction}
{Let $\ell$ be a positive integer, $0 < \tau_1 < \tau_2 < ... < \tau_l$ be some positive numbers, and $A_i$, $i=0,...,l$ some matrices in $\mathcal{M}_{n \times n}(\mathbb{R})$.} Consider the system
\begin{equation}  \label{retarded lin}
    \textstyle\dot{x}(t) = A_0 x(t) + \sum_{j=1}^{\ell} A_j x(t-\tau_j)
\end{equation}
and its characteristic matrix~\cite{hale2013introduction} given by
\[
\textstyle \Delta(z) = z I - A_0 - \sum_{j = 1}^\ell A_j e^{-z \tau_j}, \qquad z \in \mathbb C.
\]
Setting $s_0 = \sup\{\Re(z) : \det \Delta(z) = 0\}$, then for any ${p} > s_0$, there exists a constant $M = M({p})$ such that, for all initial conditions $\phi$, the solution $x(\phi)$ of system \eqref{retarded lin} satisfies the exponential estimate: $
|x(\phi)(t)| \le M e^{{p} t} \|\phi\|,~ t \ge 0. $
\end{proposition}

The number $s_0$ is called the spectral abscissa of~\eqref{retarded lin}. For reservoir design, tuning $s_0$ is critical: too negative, and the reservoir loses memory of past inputs; too close to zero, and the fading memory property is compromised. In the scalar single delay-case, the roots of the characteristic matrix can be expressed in terms of the $W$-Lambert function~\cite{corless1996lambert}. Explicit algebraic conditions to assign the spectral abscissa of a scalar linear TDS can be found in the literature for specific configurations~\cite{Boussaada2022Generic, Schmoderer2024Insights}.

Next, we derive an explicit lower bound for the separation property in a single-delay linear reservoir. Consider $\tau>0$, $a_0 \in \mathbb{R}$ and $a_1 \in \mathbb{R}$. Consider the linear system
\begin{equation} \label{1d}
        \dot x(t)  = a_0 x(t) + a_1 x(t-\tau) + u(t),
\end{equation}
with the initial condition $\phi \in C_p([-\tau,0),\mathbb{R})$ and $u \in C_p(\mathbb{R}_+,\mathbb{R})$. Consider two different inputs $u$ and $v$ and the associated solutions $x^u$ and $x^v$. By linearity, the state difference $z = x^{u} - x^{v}$ is solution of \eqref{1d} with the term $u$ being replaced by $w \coloneqq u-v$ and initial condition $\phi \equiv 0$.

\begin{proposition} \label{Prop_Sep}
Denote by $\alpha_k$, $k \in \mathbb Z$, the coefficients of the Fourier expansion of $w$ on $[0, t_1]$.
Assume that the input-free system $\dot x(t) = a_0 x(t) + a_1 x(t-\tau)$
is asymptotically stable with spectral abscissa $s_0$. Then, for any $\epsilon>0$, there exists some $M_1> 0$ depending on $\epsilon$ such that for every $0<t_0<t_1$, the global separation measure on $[t_0,t_1]$ satisfies
\[
\textstyle \int_{t_0}^{t_1} |z(t)|^2 dt \;\ge\; (t_1-t_0) \big( \sum_{k\in\mathbb{Z}}
\frac{|\alpha_k|^2}{\Delta_k} \, \, - \, \, M_1 e^{(s_0+\epsilon) t_0} \big),
\]
where
\[
\textstyle \Delta_k = \left(a_0 + a_1\cos\frac{2k\pi\tau}{t_1}\right)^2 + \left(\frac{2k\pi \tau}{t_1} + a_1\sin\frac{2k\pi\tau}{t_1}\right)^2.
\]
\end{proposition}

\begin{proof}
Due to space restriction, we only give a sketch of the proof. We have $\textstyle w(t) = \sum_{k\in\mathbb{Z}} \alpha_k e^{\frac{2 i k \pi t}{t_1}}.$ We then introduce the change of variables
\[
\textstyle\tilde z(t) = z(t) + \sum_{k\in\mathbb{Z}} \beta_k e^{\frac{2 i k \pi t}{t_1}}.
\]
The goal is to choose the $\beta_k$ so that $\tilde z$ satisfies the homogeneous delayed equation $\dot{\tilde z}(t) = a_0 \tilde z(t) + a_1\tilde z(t-\tau).$
This yields $k \in \mathbb Z$,
\[
\beta_k \left( i \omega_k - a_0 - a_1 e^{-i \omega_k \tau} \right) = \alpha_k.
\]
We obtain \(|\beta_k|^2 = \frac{|\alpha_k|^2}{\Delta_k}\), where $\Delta_k = |i \omega_k - a_0 - a_1 e^{-i \omega_k \tau}|^2$. {We now fix $t_0 \in [0, t_1)$ and compute the energy of the state difference by expanding $\int_{t_0}^{t_1} |z(t)|^2 dt=\textstyle\int_{t_0}^{t_1} \left| \tilde z(t) - \sum_{k\in\mathbb{Z}} \beta_k e^{i \omega_k t} \right|^2 dt$.} By Proposition \ref{cv rate}, the homogeneous solution $\tilde z(t)$ satisfies $|\tilde z(t)|\leq M\mathrm{e}^{(s_0 + \epsilon)t}||\tilde z_0||$. Thus, its squared integral and the cross-terms with the bounded periodic signal are subsumed into a bounded decaying term $M e^{(s_0+\epsilon)t_0}$. Finally, exploiting the orthogonality of the Fourier modes on the remaining periodic part over $t \ge t_0$ yields the $(t_1-t_0)\sum_k |\beta_k|^2$ term, concluding the proof.
\end{proof}

By neglecting the transient term $\tilde z$, we obtain a metric that is independent of the initial condition.

As mentioned in Section~\ref{math}, separation due to noise should be excluded from the bound. Assuming that frequency components with index $|k| > k_0$ arise from noise, a good reservoir should maximize $\sum_{|k| \le k_0} \frac{|\alpha_k|^2}{\Delta_k}$ while keeping $\sum_{|k| > k_0} \frac{|\alpha_k|^2}{\Delta_k}$ small. For a fixed spectral abscissa $s_0$, tuning the reservoir parameters to minimize $\Delta_k$ over the signal band reduces to a constrained optimization problem, which can be solved in the single-delay case by using, for instance, \cite[Th.~17]{Boussaada2022Generic} or \cite[Th.~8]{Schmoderer2024Insights}. This analysis extends naturally to multiple delays. The quantity $\Delta_k$ is generalized by replacing $a_1 \cos\frac{2k\pi\tau}{t_1}$ and $a_1\sin\frac{2k\pi\tau}{t_1}$ with their multi-delay counterparts $\sum_j a_j \cos\frac{2k\pi\tau_j}{t_1}$ and $\sum_j a_j \sin\frac{2k\pi\tau_j}{t_1}$. Each additional delay introduces free  parameters that can be used to minimize $\Delta_k$ while keeping the  spectral abscissa fixed. {For the case of two delays, the Multiplicity-Induced-Dominancy property \cite{fueyo2023pole} allows us to assign the spectral abscissa.} \emph{Adding delays improves separation without sacrificing stability}, a direction that will be further explored in future work.

We conclude this section by addressing the practical verification of $\delta$ISS for linear reservoirs. For linear TDS, $\delta$ISS reduces to standard ISS. In this setting, converse Lyapunov-Krasovskii results guarantee that asymptotic stability of the 
input-free system is equivalent to $\delta$ISS and to the existence of a Lyapunov functional~\cite{kharitonov2012time}. Constructive LKFs can be obtained via standard LMIs for systems 
with multiple delays~\cite{fridman2014introduction}. In the single-delay case, under the condition $a_0 + |a_1| < 0,$ a possible quadratic functional is of the form
\begin{equation}
    V(x_t) =  x(t)^2 + |a_1|\int_{-\tau}^{0} x(t+s)^2\,ds.
\end{equation}
We immediately obtain for any $\epsilon>(-2(a_0+|a_1|))^{-1}$,
\begin{align}
    D^{+}V(x_t) \leq (2a_0+2|a_1|+\tfrac{1}{\epsilon})x^2(t)+\epsilon u^2(t). \label{eq_ISS_line}
\end{align}

\section{Simulations and results} \label{sec:simulations}
This section illustrates the framework on two tasks. The goal is not to compete 
with state-of-the-art RC implementations, but to demonstrate that good performance 
is achievable with a minimalist setup, and to validate the separation analysis of 
Section~\ref{lin}.

\paragraph*{NARMA10 benchmark}
We consider the NARMA10 benchmark, a standard RC task in which the target $y_i$ 
depends recursively on the input sequence $(u_k)_{k\in\mathbb{N}}$ and the ten 
previous outputs~\cite{appeltant2012reservoir}. Performance is measured by the 
Normalized Root Mean Square Error:
\[
\mathrm{NRMSE} = \sqrt{\frac{\frac{1}{n}\sum_{i=1}^{n}(y_i - \hat{y}_i)^2}{\mathrm{Var}(y)}}.
\]
We consider a linear reservoir
\begin{equation}\label{eq lin}
    \dot{x}(t) = a_0 x(t) + a_1 x(t-\tau) + u(t) + \xi(t),
\end{equation}
where $\xi$ is Gaussian noise with standard deviation $0.001$. We tested this reservoir with $a_0 = -1$, $a_1 = 0.9e^{-0.1}$, $\tau = 1$ (spectral abscissa $s_0 = -0.1$), and parameters 
$\theta = 0.2$, $N = 10$ (implying $T = 2$), trained on 500 points and tested on 
70. We obtain $\mathrm{NRMSE} = 0.38$. In \cite{appeltant2012reservoir}, the reported value of $\mathrm{NRMSE}$ for a linear delay-based reservoir was $0.4$. In Fig.~\ref{fig:lin}, we compare the target values of the NARMA10 output sequence. For comparison, the 
nonlinear reservoir
\begin{equation}\label{eq log}
    \dot{x}(t) = -x(t) + g\big(x(t-1) + u(t)\big) + \xi(t),
\end{equation}
with $ g(t) = \mathrm{sign}(t)\ln(1+|t|)$ yields $\mathrm{NRMSE} = 0.33$. The marginal improvement over the linear case 
highlights the inherent computational richness of the infinite-dimensional linear 
dynamics. Note that increasing the number of neurons did not lead to any improvement in this specific example.

\begin{figure}[thpb]
    \centering
    \includegraphics[width=0.9\linewidth]{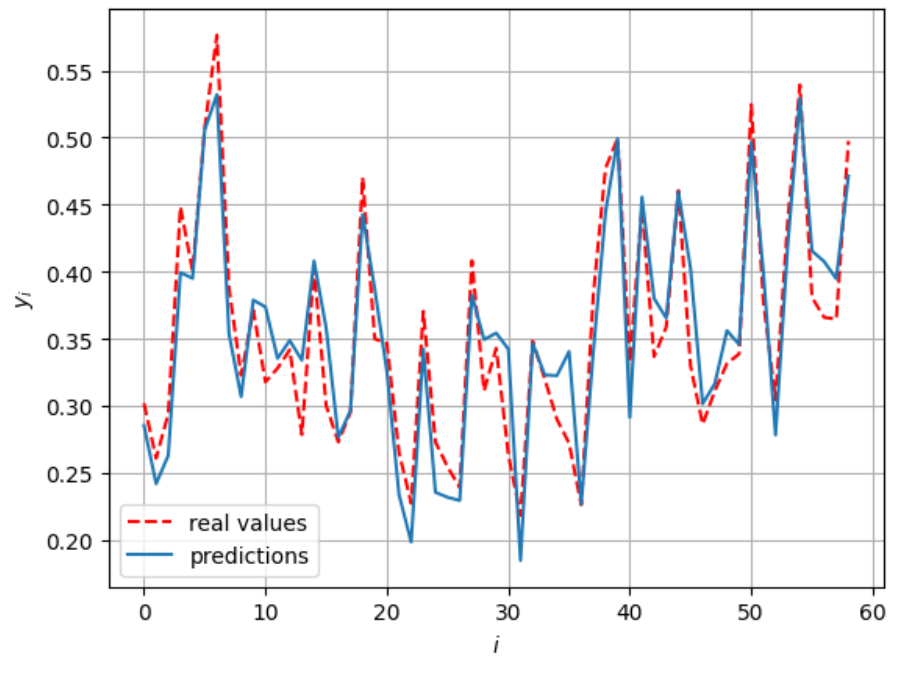}
    \caption{NARMA10 predictions for System~\eqref{eq lin} ($a_0=-1$, 
    $a_1=0.9e^{-0.1}$, $\tau=1$, $\mathrm{NRMSE}=0.38$).}
    \label{fig:lin}
\end{figure}

\paragraph*{Possible trade-off} We now compare two parameter configurations sharing the same spectral abscissa 
$s_0 = -0.1$:
\begin{itemize}
\item Config.~1: $a_0 = -1.0$,\ $a_1 = 0.9e^{-0.1}$,\ $\tau=1$.
    \item Config.~2: $a_0 = -0.5$,\ $a_1 = 0.4e^{-0.1}$,\ $\tau=1$,
\end{itemize}
For both configurations, we computed $\Delta_k^{-1}$ for $k = 1,\ldots,10$ with 
$t_1=T \in \{20,50\}$. Config.~2 yields values of $\Delta_k^{-1}$ 
approximately $80\%$ larger than Config.~1 in both cases, confirming that 
identical stability can coexist with significantly different separation 
performance. For $T=20$, both configurations led to a normalized error around $0.37$ whereas for $T=50$, we found $\mathrm{NRMSE}=0.43$ for Config.~1 and $\mathrm{NRMSE}=0.39$ for Config.~2. We also computed the critical values of $\epsilon$ in~\eqref{eq_ISS_line} for both configurations. We obtain $\epsilon^* = 2.7$ for Config.~1 and 
$\epsilon^* = 3.6$ for Config.~2 implying a possible larger ISS bound in the first case. This illustrates the trade-off between the separation property and the fading memory.  All in all, at fixed spectral abscissa, the parameters $(a_0, a_1)$ are genuine degrees of 
freedom for maximizing low-frequency separation.

\section{Discussion and perspectives}
This paper establishes a control-theoretic foundation for the analysis of time-delay reservoir computers. We have introduced formal definitions of the separation property, robustness, and fading memory, and shown that $\delta$ISS provides a unifying sufficient condition for the latter two. For linear reservoirs, we derived an explicit lower bound for the separation distance via Fourier analysis, and discussed why adding delays is a principled way to improve separation while preserving stability. Several directions remain open. First, a full nonlinear analysis via linearization techniques would extend the spectral results of Section~\ref{lin} to a broader class of reservoirs, with criteria for tuning the spectral abscissa near the imaginary axis, the hallmark of edge-of-chaos operation. Then, geometrical aspects of the separation property need to be considered to understand how the distribution of the states in the phase space impacts the computational capacity of the reservoir. Third, a quantitative analysis of memory capacity in terms of the system parameters would complement the separation bound derived here. Together, these directions outline a roadmap toward a complete mathematical theory of time-delay reservoir computing.



\section*{Appendix} \label{appendix}
We recall in this section a few definitions and standard results on TDS and Lyapunov--Krasovskii theory~\cite{chaillet2023iss}. 

\begin{definition}[Class $\mathcal{K}$, $\mathcal{K}_\infty$ function]
A continuous function $\alpha\colon [0,+\infty) \to [0,+\infty)$ is said to be of \emph{class $\mathcal{K}$} if $\alpha(0)=0$ and $\alpha$ is increasing. Moreover, $\alpha$ is called a \emph{class $\mathcal{K}_\infty$} function if $\displaystyle \lim_{r \to +\infty} \alpha(r) = +\infty$.
\end{definition}

\begin{definition}[Class $\mathcal{KL}$ function]
A continuous function $\beta\colon [0, +\infty) \times [0,+\infty) \to [0,+\infty)$ is of \emph{class $\mathcal{KL}$} if
\begin{enumerate}[label={\roman*)}]
    \item $\forall t \ge 0$, 
    $r \mapsto \beta(r,t)$ is a class $\mathcal{K}$ function,
    \item $\forall r \ge 0$, the function
    $t \mapsto \beta(r,t)$ is non-increasing,
    \item $\forall r \ge 0, \displaystyle \lim_{t \to +\infty} \beta(r,t) = 0$.
\end{enumerate}
\end{definition}

\begin{definition}[Driver derivative]
The Driver derivative of a functional $V \colon C \to \mathbb{R}$ at $x_t$ is defined as
\[
D^+V(x_t) = \limsup_{h \to 0^+} \tfrac{V(x_{t+h}) - V(x_t)}{h},
\]
\end{definition}







\bibliographystyle{abbrv}
\bibliography{biblio.bib}

\begin{thebibliography}{10}

\bibitem{amrane2018qualitative}
S.~Amrane, F.~Bedouhene, I.~Boussaada, and S.-I. Niculescu.
\newblock On qualitative properties of low-degree quasipolynomials: further
  remarks on the spectral abscissa and rightmost-roots assignment.
\newblock {\em Bull. Math. Soc. Sci. Math. Roumanie (N.S.)},
  61(109)(4):361--381, 2018.

\bibitem{appeltant2012reservoir}
L.~Appeltant.
\newblock {\em Reservoir computing based on delay-dynamical systems}.
\newblock PhD thesis, Vrije Universiteit Brussel/Universitat de les Illes
  Balears, 2012.

\bibitem{bernard2022observer}
P.~Bernard, V.~Andrieu, and D.~Astolfi.
\newblock Observer design for continuous-time dynamical systems.
\newblock {\em Annu. Rev. Control}, 53:224--248, 2022.

\bibitem{bishop1995training}
C.~M. Bishop.
\newblock Training with noise is equivalent to {T}ikhonov regularization.
\newblock {\em Neural computation}, 7(1):108--116, 1995.

\bibitem{Boussaada2022Generic}
I.~Boussaada, G.~Mazanti, and S.-I. Niculescu.
\newblock The generic multiplicity-induced-dominancy property from retarded to
  neutral delay-differential equations: When delay-systems characteristics meet
  the zeros of {K}ummer functions.
\newblock {\em C. R. Math. Acad. Sci. Paris}, 360:349--369, 2022.

\bibitem{chaillet2023iss}
A.~Chaillet, I.~Karafyllis, P.~Pepe, and Y.~Wang.
\newblock The {ISS} framework for time-delay systems: a survey.
\newblock {\em Math. Control Signals Systems}, 35(2):237--306, 2023.

\bibitem{corless1996lambert}
R.~Corless, G.~Gonnet, D.~Hare, D.~Jeffrey, and D.~Knuth.
\newblock On the {L}ambert {W} function.
\newblock {\em Advances in Computational mathematics}, 5(1):329--359, 1996.

\bibitem{dambre2012information}
J.~Dambre, D.~Verstraeten, B.~Schrauwen, and S.~Massar.
\newblock Information processing capacity of dynamical systems.
\newblock {\em Scientific reports}, 2(1):514, 2012.

\bibitem{duport2016fully}
F.~Duport, A.~Smerieri, A.~Akrout, M.~Haelterman, and S.~Massar.
\newblock Fully analogue photonic reservoir computer.
\newblock {\em Scientific reports}, 6(1):22381, 2016.

\bibitem{fridman2014introduction}
E.~Fridman.
\newblock {\em Introduction to time-delay systems: Analysis and control}.
\newblock Springer, 2014.

\bibitem{fueyo2023pole}
S.~Fueyo, G.~Mazanti, I.~Boussaada, Y.~Chitour, and S.-I. Niculescu.
\newblock On the pole placement of scalar linear delay systems with two delays.
\newblock {\em IMA J. Math. Control Inform.}, 40(1):81--105, 2023.

\bibitem{hale2013introduction}
J.~Hale and S.~Lunel.
\newblock {\em Introduction to functional differential equations}, volume~99.
\newblock Springer Science \& Business Media, 2013.

\bibitem{hoerl1970ridge}
A.~Hoerl and R.~Kennard.
\newblock Ridge regression: {B}iased estimation for nonorthogonal problems.
\newblock {\em Technometrics}, 12(1):55--67, 1970.

\bibitem{jaeger2001echo}
H.~Jaeger.
\newblock The ``{e}cho state'' approach to analysing and training recurrent
  neural networks-with an erratum note.
\newblock {\em Bonn, Germany: German national research center for information
  technology gmd technical report}, 148(34):13, 2001.

\bibitem{janny2021deep}
S.~Janny, V.~Andrieu, M.~Nadri, and C.~Wolf.
\newblock Deep {KKL}: Data-driven output prediction for non-linear systems.
\newblock In {\em 2021 60th IEEE Conference on Decision and Control (CDC)},
  pages 4376--4381, 2021.

\bibitem{karafyllis2011stability}
I.~Karafyllis and Z.-P. Jiang.
\newblock {\em Stability and stabilization of nonlinear systems}.
\newblock Springer Science \& Business Media, 2011.

\bibitem{kharitonov2012time}
V.~Kharitonov.
\newblock {\em Time-delay systems: {L}yapunov functionals and matrices}.
\newblock Springer Science \& Business Media, 2012.

\bibitem{koster2021insight}
F.~K{\"o}ster, S.~Yanchuk, and K.~L{\"u}dge.
\newblock Insight into delay based reservoir computing via eigenvalue analysis.
\newblock {\em Journal of Physics: Photonics}, 3(2):024011, 2021.

\bibitem{larger2017high}
L.~Larger, A.~Bayl{\'o}n-Fuentes, R.~Martinenghi, V.~Udaltsov, Y.~Chembo, and
  M.~Jacquot.
\newblock High-speed photonic reservoir computing using a time-delay-based
  architecture: {M}illion words per second classification.
\newblock {\em Physical Review X}, 7(1):011015, 2017.

\bibitem{legenstein2007edge}
R.~Legenstein and W.~Maass.
\newblock Edge of chaos and prediction of computational performance for neural
  circuit models.
\newblock {\em Neural networks}, 20(3):323--334, 2007.

\bibitem{Lukosevicius2009Reservoir}
M.~Lukoševičius and H.~Jaeger.
\newblock Reservoir computing approaches to recurrent neural network training.
\newblock {\em Computer Science Review}, 3(3):127--149, 2009.

\bibitem{maass2002real}
W.~Maass, T.~Natschl{\"a}ger, and H.~Markram.
\newblock Real-time computing without stable states: {A} new framework for
  neural computation based on perturbations.
\newblock {\em Neural computation}, 14(11):2531--2560, 2002.

\bibitem{Michiels2014Stability}
W.~Michiels and S.-I. Niculescu.
\newblock {\em Stability, control, and computation for time-delay systems},
  volume~27 of {\em Advances in Design and Control}.
\newblock Society for Industrial and Applied Mathematics (SIAM), Philadelphia,
  PA, second edition, 2014.
\newblock An eigenvalue-based approach.

\bibitem{pepe2006lyapunov}
P.~Pepe and Z.-P. Jiang.
\newblock A {L}yapunov--{K}rasovskii methodology for {ISS} and i{ISS} of
  time-delay systems.
\newblock {\em Systems Control Lett.}, 55(12):1006--1014, 2006.

\bibitem{Schmoderer2024Insights}
T.~Schmoderer, I.~Boussaada, S.-I. Niculescu, and F.~Bedouhene.
\newblock Insights on equidistributed real spectral values in second-order
  delay systems: perspectives in partial pole placement.
\newblock {\em Systems Control Lett.}, 185:Paper No. 105728, 9 pp., 2024.

\bibitem{yeganefarinput}
N.~Yeganefar, P.~Pepe, and M.~Dambrine.
\newblock Input-to-state stability and exponential stability for time-delay
  systems: further results.
\newblock In {\em 2007 46th IEEE Conference on Decision and Control}, pages
  2059--2064, 2007.

\end{thebibliography}

\end{document}